\documentclass[11pt]{amsart}

\usepackage{amsfonts}
 \usepackage{amsfonts,amssymb}

\usepackage{mathrsfs}
\let\mathcal\mathscr

\usepackage[all,ps,cmtip]{xy}

\def\R{{\bf R}}

\def\D{{\bf D}}

\def\llra{\hbox to 10mm{\rightarrowfill}}

\def\lllra{\hbox to 15mm{\rightarrowfill}}

\def\PA{{\widehat A}}
\def\PB{{\widehat B}}
\def\PK{{\widehat K}}

\def\pf{{\widehat f}}
\def\phi{{\varphi}}

\def\cF{\mathcal{F}}

\def\cO{\mathcal{O}}

\def\cP{\mathcal{P}}

\def\cM{\mathcal{M}}

\def\cQ{\mathcal{Q}}

\def\cW{\mathcal{W}}

\let\tilde\widetilde

\def\eps{\varepsilon}

\def\Id{\hbox{\rm Id}}
\def\Im{\mathop{\rm Im}\nolimits}

\DeclareMathOperator{\rank}{rank}

\DeclareMathOperator{\codim}{codim}
\DeclareMathOperator{\Pic}{Pic}
\DeclareMathOperator{\Hom}{Hom}
\def\Im{\mathop{\rm Im}\nolimits}
\DeclareMathOperator{\Supp}{Supp}

\DeclareMathOperator{\Spec}{Spec}

\newtheorem{lemm}{Lemma}[section]
\newtheorem{theo}[lemm]{Theorem}
\newtheorem{coro}[lemm]{Corollary}
\newtheorem{prop}[lemm]{Proposition}

\newtheorem*{conj*}{Conjecture}

\theoremstyle{definition}
\newtheorem{defi}[lemm]{Definition}
\newtheorem{rema}[lemm]{Remark}

\newtheorem{conj}[lemm]{Conjecture}

\theoremstyle{remark}
\newtheorem*{remark*}{Remark}
\newtheorem*{note*}{Note}

\begin{document}
\title[vanishing holomorphic Euler characteristic II]{Varieties with vanishing holomorphic Euler characteristic II}
\author[J. Chen]{Jungkai Alfred Chen}
\address{National Center for Theoretical Sciences, Taipei Office\\
 and Department of Mathematics\\1 Sec. 4, Roosevelt Rd. Taipei 106, Taiwan}
\email{{\tt jkchen@math.ntu.edu.tw}}
\author[Z. Jiang]{Zhi Jiang}
\address{Math\'ematiques B\^{a}timent 425\\
Universit\'{e} Paris-Sud\\
F-91405 Orsay, France}
\email{zhi.jiang@math.u-psud.fr}

\begin{abstract} We continue our study on smooth complex projective varieties $X$ of maximal Albanese dimension and of general type  satisfying  $\chi(X, \cO_X)=0$.  We formulate a conjectural characterization of such varieties and prove this conjecture when the Albanese variety has only three simple factors.
 \end{abstract}
\maketitle

\section{Introduction}

The purpose  of this paper is to study the birational geometry of
varieties $X$ of general type and of maximal Albanese dimension with
$\chi(\omega_X)= 0$.

In recent years,  these varieties have attracted considerable
attention. Green and Lazarsfeld showed  in \cite{gl1} that  a variety of maximal Albanese dimension satisfies $\chi(X, \omega_X)\ge 0$.
It was conjectured by Koll\'ar \cite[18.12.1]{K1}
that a variety of general type and maximal Albanese dimension
would satisfy $\chi(\omega_X)> 0$. A couple years later,
Ein-Lazarsfeld
disprove the conjecture by providing an example of threefold of
general type and maximal Albanese dimension with $\chi(\omega_X) =
0$.

In fact, in the recent studies on the structure of the
pluricanonical maps and of the Iitaka map, it has been realized that
the case $\chi(\omega_X) = 0$ is usually the hardest case. For example, it was shown in \cite{CH} that the tricanonical map is birational for varieties of general type and maximal Albanese dimension with $\chi(\omega_X) >0$. However, if we assume $\chi(\omega_X)=0$ instead, then it is more difficult to prove that the tricanonical map is birational (see \cite{JLT}).

It is thus natural and important to characterize or classify  varieties of general type and maximal Albanese dimension with $\chi(\omega_X)=0$ explicitly. Our previous joint work with Olivier Debarre was the starting point toward this direction, in which we prove a characterization in dimension three. Since the characterizing properties are preserved under birational maps and finite \'etale maps, it was shown that the Albanese variety has at least three simple factors and the example of Ein and Lazarsfeld is the only possibility in dimension three.

The main result in this article is to prove a similar characterization in higher dimensions assuming that the Albanese variety has three simple factors.

\begin{theo}\label{maintheorem}
Let $X$ be a variety of general type and of maximal Albanese dimension and
assume that $A_X$ has only three simple factors. If $\chi(X, \omega_X)=0$, there exist
simple abelian varieties $K_1$, $K_2$, $K_3$, double coverings
from normal varieties
 $F_i\rightarrow K_i$ with associated involution $\tau_i$, and an isogeny
 $\eta \colon K_1 \times K_2 \times K_3 \to A_X$, such that the base change $\tilde X$ is birational
 to $$(F_1\times F_2\times F_3)/\langle \sigma \rangle$$
 where $\sigma=\tau_1\times\tau_2\times\tau_3$ is the diagonal involution.
 That is, we have the following commutative diagram:

$$
\xymatrix@C=15pt
 {&(F_1\times F_2\times F_3)/\langle \sigma \rangle \ar@{->>}[dr]\\
\widetilde X\ar@{->>}[rr]^{a_{\widetilde X}}\ar@{}[drr]|\square\ar@{->>}[ur]^-\eps\ar@{->>}[d]&& K_1\times K_2\times K_3\ar@{->>}[d]^\eta\\
X\ar@{->>}[rr]_{a_X} &&A_X. }
$$
where  $\eps$ is a desingularization.
\end{theo}

The birational geometry of varieties of maximal Albanese dimension is governed by cohomological support loci of the push-forward of the canonical sheaf $a_{X*} \omega_X$. It is well-known that $a_{X*} \omega_X$ is a GV-sheaf.
The technical advance of this article is indeed the following decomposition theorem (see Theorem \ref{decomposition} and Theorem \ref{decompositionwithtorsion} for  details), which implies that $a_{X*}\omega_X$ is not far from being M-regular.

\begin{theo}Let $f: X\rightarrow A$ be a generically finite morphism to an abelian variety. Then, we have $$f_*\omega_X\simeq \bigoplus_i (p_i^*\cF_i\otimes P_i),$$
where $p_i: A\rightarrow A_i$ are quotients of abelian varieties, $\cF_i$ are M-regular sheaves on $A_i$, and $P_i$ are torsion line bundles on $A$.
\end{theo}
\begin{rema}In the above formulation, we allow $p_i: A\rightarrow A_i$ to be trivial quotients, namely $p_i$ could be an isomorphism or a fibration  to $\Spec\mathbb{C}$.
\end{rema}

This decomposition theorem can certainly be applied to prove globally generated properties for canonical or pluricanonical bundles. Another important
application is a criterion for birationality of morphisms between varieties of maximal Albanese dimension, which is a main ingredient of the proof of Theorem \ref{maintheorem}.

The paper is organized as follows. In section 2 we introduce definitions and prove some basic results on Fourier-Mukai transform of GV sheaves. Section 3 is devoted
to prove the decomposition theorem and its corollaries. In section 4 we provide several birational criterion of morphisms between varieties of maximal Albanese dimension.
In section 5 we study the general structure of varieties $X$ of general type and of maximal Albanese dimension with $\chi(X, \omega_X)=0$. We formulate
a conjectural characterization of such varieties. Finally, in section 6,  we restrict ourselves to the case when $A_X$ has only three simple factors and
prove Theorem \ref{maintheorem}.

 \medskip\noindent{\bf Acknowledgements.} This work started during the second author's visit to NCTS (Mathematics Division, Taipei Office). The second author
 thanks NCTS for their warm hospitality and the excellent research atmosphere. The authors thank Olivier Debarre for numerous conversations on this subject.

\section{Notation and Preliminaries}
For any smooth projective variety $X$, we will denote by $a_X: X\rightarrow A_X$ the Albanese morphism of $X$ and $\PA_X=\Pic^0(X)$ the dual of the
 Albanese variety. We will denote by $\D(X)$ the bounded derived  category of coherent sheaves on $X$. Following \cite{PP-GV}, for any object $\mathcal{E}\in \D(X)$,
we write $\R\Delta(\mathcal{E}):=\R\mathcal{H}om(\mathcal{E}, \omega_X)$.

For an abelian variety $A$ and its dual $\PA$, we denote by $\cP_A$ the normalized Poincar\'e line bundle on $A\times \PA$. For $\alpha\in \PA$, we denote by $P_{\alpha}$ the
line bundle that represents $\alpha$.
By \cite{muk2}, the following functors give equivalence between $\D(A)$ and $\D(\PA)$:
\begin{eqnarray*}
\R\Phi_{\cP_A}: \D(A)\rightarrow \D(\PA), \;\; \R\Phi_{\cP_A}(\cdotp)=\R p_{\PA*}(p_{A}^*(\cdotp)\otimes \cP_A),\\
\R\Psi_{\cP_A}: \D(\PA)\rightarrow \D(A), \;\; \R\Psi_{\cP_A}(\cdotp)=\R p_{A*}(p_{\PA}^*(\cdotp)\otimes \cP_A).
\end{eqnarray*}

For any coherent sheaf $\cF$ on $X$ and any morphism $f: X\rightarrow A$ to an abelian variety, we define the $i$-th cohomological locus
$$V^i(\cF, f):=\{\alpha \in \PA\mid H^i(X, \cF\otimes P_{\alpha})\neq 0\}.$$ If $f=a_X$ is the Albanese morphism, we will simply denote by $V^i(\cF)$
the $i$-th cohomological locus.

For an abelian variety $A$ and its dual $\PA$, we always use the notation $\widehat{*}$ to denote an abelian subvariety of $\PA$, and then
$*=\widehat{\widehat{*}}$ is the natural quotient of $A$.

We recall the definition of GV-sheaves and M-regular sheaves on abelian varieties (see \cite{reg1} and \cite{PP-GV}).
\begin{defi}
 Let $\cF$ be a coherent sheaf on an abelian variety $A$. Then $\cF$ is a GV-sheaf if
$$\codim_{\PA}\Supp \R^i\Phi_{\cP_A}(\cF)\geq i$$ for all $i\geq0$; $\cF$ is a M-regular sheaf if
$$\codim_{\PA}\Supp \R^i\Phi_{\cP_A}(\cF)> i$$ for all $i>0$.
\end{defi}
\begin{rema}\label{properties}
 The following properties of M-regular sheaves and GV-sheaves are quite useful:
\begin{itemize}
\item[1)] if $\cF$ is a GV-sheaf (resp. M-regular sheaf) on $A$, then
$$\widehat{\R\Delta(\cF)}:=\R\Phi_{\cP_A}(\R\Delta(\cF))[g]$$ is a coherent (resp. torsion-free) sheaf of $\PA$ supported on $V^0(\cF)$ (\cite[Theorem 1.2]{Hac}, \cite[Proposition 2.8]{reg3});
\item[2)] if $\cF$ is a GV-sheaf, then $$\mathcal{E}xt^i(\widehat{\R\Delta(\cF)}, \cO_{\PA})\simeq (-1_{\PA})^*\R^i\Phi_{\cP_A}(\cF)$$ (see \cite[Remark 3.13]{PP-GV});
\item[3)] $\cF$ is a GV-sheaf (resp. M-regular sheaf) if and only if $\codim_{\PA}V^i(\cF)\geq i$ (resp. $>i$) for all $i\geq 1$ (\cite[Lemma 3.6]{PP-GV}).
\end{itemize}
\end{rema}

The following proposition is a generalization of \cite[Proposition 14.7.9 (b)]{BL}
\begin{prop}\label{fm}Let $f: A\rightarrow B$ be a quotient of abelian varieties. Assume that $\dim A=g$ and $\dim B=g_1$. Then
\begin{itemize}
\item[(1)]
we have the natural
isomorphism of functors from $\D(B)$ to $\D(\PA)$:
$$\R\Phi_{\cP_A}\circ f^*\circ [g]\simeq \widehat{f}_*\circ \R\Phi_{\cP_{B}}\circ [g_1];$$
\item[(2)]
and the natural isomorphism of functors from $\D(\PB)$ to $\D(A)$:
$$f^*\circ \R\Psi_{\cP_B}\simeq \R\Psi_{\cP_A}\circ \pf_*$$
\end{itemize}
\end{prop}

\begin{rema}
Note that the direct image $\widehat{f}_*$ needs not to be derived for  a closed embedding and similarly, $f^*$ needs not to be derived for a smooth morphism.
\end{rema}
\begin{proof}
We note that $(2)$ is equivalent to $(1)$ by Fourier-Mukai equivalence. It suffices to prove (1).

For brevity, we will abuse the notation of pull-back and push-forward with its derived functors.
We consider the following commutative diagram:

\begin{eqnarray*}
\xymatrix{\ar@{}[drr]|{\square}
A\ar[d]_{f} &&\ar@{}[dr]|{\circlearrowleft}  A \times \PA\ar[d]_{f_\PA}\ar[ll]_{p_A}\ar[drr]^{p_\PA} && \\
 B && \ar@{}[dl]|{\circlearrowleft} \ar@{}[drr]|{\square} B \times \PA\ar[ll]_{\pi_B}\ar[rr]^{\pi_{\PA}}  && \PA \\
 && B \times \PB\ar[llu]^{p_B}\ar[rr]^{p_\PB}\ar[u]_{\widehat{f}_B} && \PB\ar[u]_{\widehat{f}},
}
\end{eqnarray*}

where each map is either the natural projection, dual map, or base change.
We then have
\begin{eqnarray*}\R\Phi_{\cP_A}(f^*\cF)&=& {p_{\PA}}_* (p_{A}^*f^*\cF\otimes \cP_A)\\
&\simeq &  {p_{\PA}}_* (f_{\PA}^* \pi_B^*\cF\otimes \cP_A)\\
&\simeq &  {\pi_{\PA}}_*  {f_{\PA}}_* (f_{\PA}^* \pi_B^*\cF\otimes \cP_A)\\
&\simeq &  {\pi_{\PA}}_* ( \pi_B^*\cF\otimes  {f_{\PA}}_* \cP_A) \;\;\;\;\; \textrm{(projection formula).}
\end{eqnarray*}
Similarly,
\begin{eqnarray*}
 \widehat{f}_*\R\Phi_{\cP_{B}}(\cF))&=& \widehat{f}_* {p_{\PB}}_* ( p_B^* \cF \otimes \cP_B)\\
&\simeq & {\pi_{\PA}}_* \widehat{f}_{B*} ({\widehat{f}_B}^{*} \pi_B^* \cF \otimes \cP_B)\\
&\simeq & {\pi_{\PA}}_* ( \pi_B^* \cF \otimes \widehat{f}_{B*} \cP_B) \;\;\;\;\; \textrm{(projection formula).}
\end{eqnarray*}
Thus if $f$ has connected fibers, we conclude the proof by  Lemma \ref{lemmaPoincare} below.

If $f$ has disconnected fibers, we consider the Stein factorization  of $f$: $$A\xrightarrow{g} B'\xrightarrow{\pi} B,$$
where $\pi$ is an isogeny between abelian varieties and $g$ is a fibration.

Let $\cF\in\D(B)$. Since $\pi$ is an isogeny, by \cite[Proposition 14.7.9 (b)]{BL}, we have $$\R\Phi_{\cP_{B'}}(\pi^*\cF)\simeq \widehat{\pi}_*\R\Phi_{\cP_{B}}(\cF).$$
Since $g$ is a fibration, we also have $$\widehat{g}_*\R\Phi_{\cP_{B'}}(\pi^*\cF)[g_1]\simeq \R\Phi_{\cP_{A}}(g^*\pi^*\cF)[g].$$
Thus we have \begin{eqnarray*}
\widehat{f}_*\R\Phi_{\cP_{B}}(\cF)[g_1]&\simeq& \widehat{g}_*\widehat{\pi}_*\R\Phi_{\cP_{B}}(\cF)[g_1]\\
&\simeq& \widehat{g}_*\R\Phi_{\cP_{B'}}(\pi^*\cF)[g_1]\\
&\simeq& \R\Phi_{\cP_{A}}(g^*\pi^*\cF)[g]\\
&\simeq& \R\Phi_{\cP_{A}}(f^*\cF)[g].
\end{eqnarray*}
This completes the proof.
\end{proof}

\begin{lemm}\label{lemmaPoincare}Let $f: A\rightarrow B$ be a quotient of abelian varieties with connected fibers. Keeping the notation as in Proposition \ref{fm}, we have $$\R {f_{\PA}}_* \cP_A \circ [g] \simeq  {\widehat{f}_{B*}} \cP_B \circ [g_1] \in \D(B\times\PA).$$
\end{lemm}
\begin{proof}
Let $K$ be the kernel of $f$.

\medskip
\noindent{\bf Step 1.}
{\em We first assume that $A=B\times K$ and $f: A\rightarrow B$ is the natural projection.}

Then $\cP_A=p_{13}^*\cP_{B}\otimes p_{24}^*\cP_{K}$ on $A\times \PA=B\times K\times \PB\times\PK$ and $\widehat{f}: \PB\rightarrow \PB\times\PK$ is the
closed embedding $x\rightarrow (x, 0_{\PK})$ for $x\in \PB$.
We consider
\begin{eqnarray*}
\xymatrix{
A\times\PA=B\times K\times \PB\times\PK\ar[d]^{f_{\PA}=(p_1,p_3, p_4)}\ar[rr]^(.7){p_{24}} && K\times \PK\ar[d]^{p_{\PK}}\\
B\times\PA=B\times\PB\times\PK\ar[d]^{p_B}\ar[rr]^(.7){q_{\PK}} &&\PK\\
B\times\PB\ar@/^1pc/[u]^{\widehat{f}_B}}
\end{eqnarray*}
Therefore,
\begin{eqnarray*} f_{\PA_*} \cP_A&=& f_{\PA*} ( p_{13}^*\cP_{B}\otimes   p_{24}^*\cP_{K})\\
 &=& f_{\PA*} ( f_{\PA}^* p_{B}^*\cP_{B}\otimes   p_{24}^*\cP_{K})\\
 &=& p_B^*\cP_B\otimes f_{\PA*}   p_{24}^*(\cP_{K})\\
&=& p_B^*\cP_B\otimes q_{\PK}^*  p_{\PK*}(\cP_{K})\\
&=& p_B^*\cP_B\otimes q_{\PK}^*\mathbb{C}_{0_{\PK}}[g_1-g]\\
&\simeq & (\widehat{f}_{B*} \cP_B)[g_1-g].
\end{eqnarray*}

\noindent{\bf Step 2.} We assume that $f: A \to B$ has connected fibers.
By Poincar\'{e}'s reducibility  theorem, we can take $\pi_B: \widetilde{B}\rightarrow B$ to be an isogeny such that the fiber product $\widetilde{f}: \widetilde{A}:=A \times_{B}\widetilde{B}\rightarrow \widetilde{B}$ is isomorphic to the projection $\widetilde{B}\times K\rightarrow  \widetilde{B}$.

We then have the following commutative diagram:
\begin{eqnarray*}
\xymatrix{
\widetilde{A}\times \widehat{\widetilde{A}}\ar[d]^{f_{\widehat{\widetilde{A}}}} && \widetilde{A}\times\widehat{A}\ar[d]^{h_A}\ar[ll]_{\pi_1}\ar[rr]^{\tau_1}&& A\times\PA\ar[d]^{f_{\PA}}\\
\widetilde{B}\times \widehat{\widetilde{A}} && \widetilde{B}\times \PA\ar[ll]_{\pi_2}\ar[rr]^{\tau_2} && B\times\PA\\
\widetilde{B}\times \widehat{\widetilde{B}}\ar@{^{(}->}[u]^{\widehat{f}_{\widetilde{B}}} && \widetilde{B}\times\PB\ar@{^{(}->}[u]^{h_B}\ar[ll]_{\pi_3}\ar[rr]^{\tau_3} && B\times\PB\ar@{^{(}->}[u]_{\widehat{f}_B}.}
\end{eqnarray*}


By \cite[(14.2)]{BL}, one sees that $\pi_1^* \cP_{\widetilde{A}} = \tau_1^* \cP_A$ and $\pi_3^* \cP_{\widetilde{B}} = \tau_3^* \cP_B$.

Together with  base change and Step 1, it is easy to verify that
\begin{eqnarray*}
\tau_2^* f_{\PA*} \cP_A [g] & \simeq \pi_2^* f_{\widehat{\widetilde{A}}*} \cP_{\widetilde{A}}[g] \\
& \simeq  \pi_2^* \widehat{f}_{{\widetilde{B}}*} \cP_{\widetilde{B}}[g_1] \\
& \simeq  \tau_2^* \widehat{f}_{{{B}}*} \cP_{B}[g_1]. \\
\end{eqnarray*}

Thus $f_{\PA*} \cP_A [g-g_1]$ is a coherent sheaf supported on the image of $\widehat{f}_{{{B}}}$ and of rank $1$ on each point of its support.
Therefore, we have $$ f_{\PA*} \cP_A\simeq \big(\widehat{f}_{B*} \cP_B\otimes L\big)[g_1-g]$$ for some $L\in \Im(\Pic^0(B)\rightarrow\Pic^0(B\times\PA))$.
Finally, by the fact that $\R p_{\PA*}\cP_A=\mathbb{C}_{0_{\PA}}[-g]$, we have $L\cong \cO_{B\times\PA}$.
\end{proof}

\begin{prop}\label{pureness}Let $f: A\rightarrow B$ be a quotient of abelian varieties and let $\cF$ be a GV-sheaf on $B$.
Then $\widehat{\R\Delta(f^*\cF)}:=\R\Phi_{\cP_A}(\R\Delta(f^*\cF))[g]$ is a coherent sheaf supported on $\widehat{f}(\PB)$. Moreover,
if $\cF$ is M-regular on $B$, then $\widehat{\R\Delta(f^*\cF)}$ is a pure sheaf supported on $\widehat{f}(\PB)$.
\end{prop}
\begin{proof}
Since $\cF$ is GV-sheaf, $\widehat{\R\Delta(\cF)}:=\R\Phi_{\cP_{B}}(\R\Delta(\cF))[g_1]$ is a coherent sheaf on $\PB$ by Remark \ref{properties}.1.
By Lemma \ref{fm}, it follows that $$\R\Phi_{\cP_A}(\R\Delta(f^*\cF))[g]=\R\pf_*(\widehat{\R\Delta(\cF)})=\pf_*(\widehat{\R\Delta(\cF)})$$ is also a coherent sheaf
supported on $\pf(\PB)$. If $\cF$ is M-regular on $B$, then $\widehat{\R\Delta(\cF)}$ is torsion-free by  Remark \ref{properties}.1.
Hence $\widehat{\R\Delta(f^*\cF)}$ is a pure sheaf supported on $\pf(\PB)$.
\end{proof}

\begin{coro}\label{Homzero}Let $f: A\rightarrow B$ be a quotient of abelian varieties. Let $\cF_1$ be a M-regular sheaf on $B$ and
let $\cF_2$ be a GV-sheaf on $A$. Assume that $\pf(\PB)$ is not contained in $V^0(\cF_2)$, then $\Hom_{A}(f^*\cF_1, \cF_2)=0$.
\end{coro}
\begin{proof}
We have
\begin{eqnarray*}
&&\Hom_{A}(f^*\cF_1, \cF_2)=\Hom_{\D(A)}(f^*\cF_1, \cF_2)\\
&\simeq&  \Hom_{\D(A)}(\R\Delta(\cF_2),  \R\Delta(f^*\cF_1))\\
&\simeq&  \Hom_{\D(\PA)}(\widehat{\R\Delta(\cF_2)}, \widehat{\R\Delta(f^*\cF_1)})\;\;\textrm{by \cite[Corollary 2.5]{muk2}}\\
&=& \Hom_{\PA}(\widehat{\R\Delta(\cF_2)}, \widehat{\R\Delta(f^*\cF_1)}).
\end{eqnarray*}
By Proposition \ref{pureness} and Remark \ref{properties},
$\widehat{\R\Delta(f^*\cF_1)}$ is a pure sheaf supported on $\pf(\PB)$ and $\widehat{\R\Delta(\cF_2)}$ is supported on $V^0(\cF_2)$.
 Since $\pf(\PB)$ is not contained in $V^0(\cF_2)$, we have $$\Hom_{\PA}(\widehat{\R\Delta(\cF_2)}, \widehat{\R\Delta(f^*\cF_1)})=0.$$ We conclude the proof.
\end{proof}

\section{A decomposition theorem for $f_*\omega_X$}
Let $f: X\rightarrow A$ be a generically finite morphism onto its image and $A$ is an abelian variety. In this section we study the sheaf $f_*\omega_X$.

By the work of Green-Lazarsfeld (\cite{gl2}) and
 Simpson \cite{sim}, we know that $V^i(\omega_X,f)=V^i(f_*\omega_X)$ is a union of torsion translated abelian subvarieties of $\PA$.
 Moreover, each irreducible component of
$V^i(\omega_X, f)$ in $\PA$ is of codimension $\geq i$ (\cite{gl1}). In particular $f_*\omega_X$ is a GV-sheaf on $A$.
However $f_*\omega_X$ often fails to be M-regular, namely $V^i(f_*\omega_X)$ often has an irreducible component of codimension $i$ in $\PA$.
In particular, if $f$ is generically finite onto $A$, then $\cO_A$ is a direct summand of  $f_*\omega_X$.
\begin{defi} Let $f: X\rightarrow A$ be a generically finite morphism to an abelian variety $A$. When $f$ is surjective onto $A$, then
we define $\cW_{X/A}$ or simply $\cW$ to be the direct summand of $f_*\omega_X$ so that $f_*\omega_X= \cO_A \oplus \cW_{X/A}$. When $f$ is not surjective onto $A$, we define $\cW_{X/A}$ to be $f_*\omega_X$.
\end{defi}

 Instead, we are interesting in the M-regularity of $\cW$.
This is actually the main
technical difficulty to study birational geometry of varieties of maximal Albanese dimension.  Hence we consider the following set which measures how far $\cW_{X/A}$
is from being M-regular.

\begin{defi} Let $\cF$ be a GV-sheaf on $A$, of which cohomological support loci consists of torsion translated subtori.
We define $S^i(\cF)$
to be the set of torsion translated subtori of $\PA$ consisting of irreducible components of  $V^i(A,\cF) $ of codimension $i$ in $\PA$.  We define
 $S(\cF):=\cup_{i=0}^{n} S^i(\cF)$ and $S(\cF)^*:=\cup_{i=1}^{n}S^i(\cF)$.
For brevity, We use $S^i_X$ (resp. $S_X$, $S_X^*$) to denote $S^i(\cW_{X/A})$ (resp. $S(\cW_{X/A})$, $S^*(\cW_{X/A}$).
\end{defi}


Therefore, by definition, $S(\cF)^*$ is empty if and only if $\cF$ is M-regular.
Note that $V^i(\cW_{X/A})=V^i(f_*\omega_X)$ unless $q(X)=\dim X$ and $i=\dim X$,
and hence every irreducible component of $S_X$ has dimension $>0$ by \cite[Theorem 3]{el}.

\noindent
{\bf Convention.} Replacing $X$ by a finite \'etale covering, we may and do assume that all irreducible component of $S_X$ pass through the origin.

For each component $\PA_k \in S^i_X$ passing through the origin, we have the dual abelian variety $A_k$ together with a surjection $p_k: A \to A_k$.  We consider a modification of the Stein factorization:
\begin{eqnarray}\label{stein}
\xymatrix{
X\ar[r]^f\ar[d]_{q_{k}}& A\ar[d]^{p_{k}}\\
Y_{k}\ar[r]^{h_{k}} & A_k}
\end{eqnarray}
where $q_{k}$ is a  fibration, $Y_{k}$ is a smooth projective variety. By \cite[Theorem 3.1]{CDJ},
we know that $R^iq_{k*}\omega_X=\omega_{Y_{k}}$ and $V^0(h_{k*}\omega_{Y_{k}})=\PA_{k}$, where $i=\codim(\PA_k, \PA)$.

\begin{lemm}\label{induction}For each $\PA_{k}\in S^i_X$ passing through the origin of $\PA$ and each $j\geq 0$, we have
 \begin{eqnarray*}
 S_{Y_{k}}^j=\{\PA \in S^{i+j}_X\mid \PA \subset \PA_{k}\}.
 \end{eqnarray*}
 \end{lemm}
\begin{proof}
 By Koll\'ar's results \cite[Proposition 7.6]{ko1} and \cite[Theorem 3.1]{ko2}, we have  $R^iq_{k*}\omega_X=\omega_{Y_{k}}$, and for $P\in\PA_k$, we have
 \begin{eqnarray*}
 H^{i+j}(X, \omega_X\otimes q_k^*h_k^*P) &\simeq & \sum_{s+t=i+j}H^s(Y_k, R^tq_{k*}\omega_X\otimes h_k^*P)\\
 &\simeq & \sum_{s+t=i+j}H^s(A_k, h_{k*}R^tq_{k*}\omega_X\otimes P).
\end{eqnarray*}
Hence $S_{Y_{k}}^j  \subseteq \{\PA \in S^{i+j}_X\mid \PA \subset \PA_{k}\}$.
On the other hand, by Hacon's generic vanishing theorem \cite[Corollary 4.2]{Hac}, $h_{k*}R^tq_{k*}\omega_X$ is a GV-sheaf on $A_k$ for each $t\geq 0$.
Thus $S_{Y_{k}}^j   \supseteq \{\PA \in S^{i+j}_X\mid \PA \subset \PA_{k}\}$.
\end{proof}

\begin{theo}\label{decomposition}Assume that each component of $S_X$ passes through the origin of $\PA$. Then for each $\PA_k \in S_X$,  there exists a nontrivial M-regular sheaf $\cF_{k}$ on
$A_{k}$, which is a direct summand of $h_{k*}\omega_{Y_{k}}$. Moreover we have an isomorphism
$$ \cW_{X/A} \simeq \bigoplus_k p_{k}^*\cF_{k}.$$
\end{theo}

\begin{proof}
For each $\PA_{k}\in S^i_X$, we use the notation in the diagram (\ref{stein}) and define $Z_{k}:=Y_{k}\times_{A_{k}}A$.
Then $Z_{k}\rightarrow Y_{k}$ is a smooth abelian fibration and $Z_{k}$ is a smooth projective variety.
We recall that  the dimension of a general fiber of $q_{k}$ is $i$ (see for instance the proof of \cite[Theorem 3]{el}).
Hence we have the natural morphism $$g_{k}: X\rightarrow Z_{k},$$ which is generically finite and surjective. Considering the natural morphisms
\begin{eqnarray}\label{diagram2}
\xymatrix{
X\ar[dr]_{q_k}\ar@/^2pc/[rr]^{f}\ar[r]^{g_{k}}& Z_{k}\ar@{}[dr]|{\square}\ar[d]^{r_k}\ar[r]^{f_{k}}& A\ar[d]^{p_k}\\
& Y_k\ar[r]^{h_k}& A_k,}
\end{eqnarray} we have
\begin{eqnarray}\label{decom1}
f_*\omega_X&=&f_{k*}g_{k*}\omega_X\nonumber\\
&=& f_{k*}(\omega_{Z_{k}}\oplus \cQ_{k})\nonumber\\
&=& p_{k}^*(h_{k*}\omega_{Y_{k}})\oplus f_{k*}\cQ_{k},
\end{eqnarray}
where the last equality holds because the right part of diagram (\ref{diagram2}) is Cartesian.

We now let $0<d_N<d_{N-1}<\cdots <d_2<d_1< n$ be the positive numbers such that $S^{d_i}_X$ is not empty.

 Note that Theorem \ref{decomposition} holds in dimension $1$. We argue by induction on $\dim X$ . Thus
 we suppose that Theorem \ref{decomposition} holds for varieties of dimension $<\dim X$.

\medskip
\noindent{\bf Step 1.} Define $\cF_i$ on $A_i$
for each $\PA_i\in S^*_X$.

Suppose $\PA_m\in S^{d_r}_X$, we know by Lemma \ref{induction} that,
\begin{eqnarray*}S_{Y_{m}}&=&S^0_{Y_{m}} \cup S^{d_{r-1}-d_{r}}_{Y_{m}} \cup  \ldots  \cup S^{d_{2}-d_{r}}_{Y_{m}} \cup S^{d_{1}-d_{r}}_{Y_{m}}  \\
&=& \{\PA_m\} \cup \{ \PA \in \cup_{j < r} S^{d_j}_X \mid  \PA \subset \PA_m\}.
\end{eqnarray*}

By induction, we know that there exists a nontrivial M-regular sheaf $\cF_m$ on $A_m$ such that $h_{m*}\omega_{Y_m}$ is a direct sum of
$\cF_m$ with sheaves pulled back from the dual of elements of $S_{Y_m}^*$.
Therefore, for $P\in \PA_m$ general, we have
\begin{eqnarray}\label{step1equality}\dim H^0(Y_m, \omega_{Y_m}\otimes h_m^*P)=\dim H^0(A_m, \cF_m\otimes P),
\end{eqnarray}
and moreover, by (\ref{decom1}), $p_m^*\cF_m$ is a direct summand of $\cW_{X/A}$.

\medskip
\noindent{\bf Step 2.} Derive the decomposition.

We start with $S^{d_1}_X$. By Step 1, for each $\PA_k\in S^{d_1}_X$, there exists a coherent sheaf $\cW_k$ on $A$ such that we have the decomposition:$$\cW_{W/A}=p_{k}^*\cF_{k}\oplus \cW_k.$$

For distinct components $\PA_{k_1},\PA_{k_2}$ in $S^{d_1}_X$, we now consider  the decomposition of identity
\begin{multline}\label{decom2}
\Id: p_{k_2}^*\cF_{k_2}\hookrightarrow \cW_{X/A}=p_{k_1}^*\cF_{k_1}\oplus f_{k_1*}\cW_{k_1}
\twoheadrightarrow p_{k_2}^*\cF_{k_2}.
\end{multline}
By Corollary \ref{Homzero}, we know that $$\Hom_A(p_{k_2}^*\cF_{k_2}, p_{k_1}^*\cF_{k_1})=0.$$ Hence $p_{k_2}^*\cF_{k_2}$ is a direct summand of $\cW_{k_1}$ and thus  $p_{k_1}^*\cF_{k_1}\oplus p_{k_2}^*\cF_{k_2} $ is a direct summand of $\cW_{X/A}$.

Continuing in this way, we finally get a decomposition:
\begin{eqnarray*}\cW_{X/A}=\bigoplus_{\PA_{k}\in S^{d_1}_X}p_{k}^*\cF_{k}\bigoplus \cW_{d_1}.
\end{eqnarray*}

Now we apply induction on $d_i$. Suppose that we have the decomposition
\begin{eqnarray*}\cW_{X/A}=\bigoplus_{\PA_{k}\in\cup_{j \leq r-1} S^{d_j}_X}p_{k}^*\cF_{k}\bigoplus \cW_{d_{r-1}},
\end{eqnarray*}
where $r\leq N$.
Then for $\PA_m\in S^{d_r}_X$, we consider \begin{multline*}
\Id: p_{m}^*\cF_{m}\hookrightarrow \bigoplus_{\PA_{k}\in\cup_{j \leq r-1} S^{d_j}_X}p_{k}^*\cF_{k}\bigoplus \cW_{d_{r-1}}
\twoheadrightarrow p_{m}^*\cF_{m},
\end{multline*}
we again conclude by (\ref{Homzero}) that $$\Hom_{A}(p_{m}^*\cF_{m}, \bigoplus_{\PA_{k}\in\cup_{j \leq r-1} S^{d_j}_X}p_{k}^*\cF_{k})=0.$$
It follows that $p_{m}^*\cF_{m}$ is indeed a direct summand of $\cW_{d_{r-1}}$.
As before, we then get the decomposition for each element of $S^{d_r}_X$: $$ \cW_{X/A}=\bigoplus_{\PA_{k}\in \cup_{j \le r} S^{d_j}_X}p_{k}^*\cF_{k}\bigoplus \cW_{d_{r}}.$$
By induction, we end up with the decomposition
\begin{eqnarray}\label{decomN} \cW_{X/A} = \bigoplus_{\PA_{k}\in S_X^*} p_{k}^*\cF_{k}\bigoplus \cW_{d_{N}}.
\end{eqnarray}

\noindent
{\bf Step 3.} Show that $\cW_{d_N}$ is either M-regular on $A$ if $V^0(f_*\omega_X)=\PA$ or trivial otherwise.

As a direct summand of $\cW_{X/A}$, $\cW_{d_N}$ is a GV-sheaf (possibly trivial). It suffices to show that $S^*(\cW_{d_N})=\emptyset$.

Assume to the contrary that  $S^{*}(\cW_{d_N})\subset S^*_X$ is not empty. We then pick $\PA_m\in S^{d_r}(\cW_{d_N})$ for some $d_r$.

For $P\in\PA_m$ general, we have
\begin{eqnarray*}h^{d_r}(A, f_*\omega_X\otimes p_m^*P)&=&h^{d_r}(A, \cW_{X/A}\otimes p_m^*P)\\
&=& \sum_{\PA_{k}\in S_X^*} h^{d_r}(A, p_{k}^*\cF_{k}\otimes p_m^*P)+h^{d_r}(A, \cW_{d_N}\otimes p_m^*P)\\
&\geq &  h^{d_r}(A, p_{m}^*(\cF_{m}\otimes P))+h^{d_r}(A, \cW_{d_N}\otimes p_m^*P)\\
&>&  h^{d_r}(A, p_{m}^*(\cF_{m}\otimes P))\\
&=& \sum_{i+j=d_r} h^i(A_m, R^{j}p_{m*}p_{m}^*(\cF_{m} \otimes P)) \\
&=&  h^0(A_m, R^{d_r}p_{m*}p_{m}^*(\cF_{m} \otimes P)) \\
&=& h^0(A_m, \cF_m\otimes P),
\end{eqnarray*}
where the second equality holds because of (\ref{decomN}) and the last two equalities holds because for each $j \geq 0$, $R^{j}p_{m*}p_{m}^*\cF_{m}$ is a direct sum of copies of $\cF_m$  and hence is M-regular on $A_m$, and in particular $R^{d_r}p_{m*}p_{m}^*(\cF_{m})=\cF_{m}$.

On the other hand, since $R^jp_{m*}f_*\omega_X$ is a GV-sheaf on $A_m$, we have
\begin{eqnarray*}
h^{d_r}(A, f_*\omega_X\otimes p_m^*P)&=& h^0(A_m, R^{d_r}p_{m*}f_*\omega_X\otimes P)\\
&=&h^0(A_m, h_{m*}R^{d_r}q_{m*}\omega_X\otimes P)\\
&=& h^0(A_m, h_{m*}\omega_{Y_m}\otimes P).
\end{eqnarray*}

Combining all the (in)equalities, we get that, for $P\in \PA_m$ general, $$h^0(A_m, h_{m*}\omega_{Y_m}\otimes P)>h^0(A_m, \cF_m\otimes P),$$
which is a contradiction to (\ref{step1equality}). Therefore, $S^*(\cW_{d_N})=\emptyset$.
\end{proof}

\begin{theo}\label{decompositionwithtorsion}Let $f: X\rightarrow A$ be a generically finite morphism. Then, for each $P_{k}+\PA_{k}\in S_X$,  there exists a nontrivial M-regular sheaf $\cF_{k}$ on
$A_{k}$ such that we have an isomorphism
$$\cW_{X/A}\simeq \bigoplus_{P_{k}+\PA_{k}\in S^*_X}P_{k}^{-1}\otimes p_{k}^*\cF_{k}.$$
\end{theo}

\begin{proof}
We take an \'etale cover $\pi: {A'}\rightarrow A$ such that, for the induced morphism ${X'} :=X\times_A {A'}\rightarrow {A'}$, every element of $S_{{X'}}$ contains
the origin of $\widehat{{A'}}$.

Note that $\cW_{{X'}/{A'}}\simeq \pi^*\cW_{X/A}$ and hence $\R\Delta\cW_{{X'}/{A'}}\simeq \pi^*\R\Delta\cW_{X/A}.$ We denote respectively by $\cM_{{X'}}$ and $\cM_X$  the sheaves $\widehat{\R\Delta(\cW_{{X'}/{A'}})}$
and $\widehat{\R\Delta(\cW_{X/A})}$.
Then, by Proposition \ref{fm}, $\cM_{{X'}} \simeq \widehat{\pi}_*\cM_X.$

By Theorem \ref{decomposition} and Remark \ref{properties}, we know that $\cM_{{X'}}$ is a direct sum of pure sheaves supported on abelian subvarieties of $\widehat{{A'}}$. Since $\widehat{\pi}$ is an isogeny between abelian varieties, $\cM_X$ is a direct summand of $\widehat{\pi}^*\widehat{\pi}_*\cM_X=\widehat{\pi}^*\cM_{{X'}}$. Hence $\cM_X$ is also a direct sum of pure sheaves supported on torsion translates of abelian subvarieties of $\widehat{A}$:
 $$\cM_X=\bigoplus_{k}\tau_{P_k}^*\iota_{k*}\cM_k,$$
 where $\tau_{P_k}$ is the translation by a torsion element $P_k\in\PA$, $\iota_k \colon \PA_k\hookrightarrow \PA$ is the embedding of an abelian subvariety of $\PA$, and $\cM_k$ is a torsion-free sheaf on $\PA_k$.

 Note that, by Fourier-Mukai duality, we have the formula:
 \begin{eqnarray}\label{dual} (-1_{A})^* \cW_{X/A}\simeq \R\Delta\R \Psi_{\cP_A}(\cM_X) \simeq \R\Delta\R \Psi_{\cP_A}( \bigoplus_{k}\tau_{P_k}^*\iota_{k*}\cM_k).
 \end{eqnarray}

 Since $\tau_{P_k}$ is a translation and by Proposition \ref{fm}, we have $$ \R \Psi_{\cP_A}(\tau_{P_k}^*\iota_{k*}\cM_k)\simeq P_k^{-1}\otimes\R \Psi_{\cP_A}(\iota_{k*}\cM_k)  \simeq P_k^{-1} \otimes p_k^*\R\Psi_{\cP_{A_k}}\cM_k.$$

 Since $p_k$ is a smooth morphism, by \cite[(3.17)]{huy}, we have
 \begin{eqnarray*}\R \Delta (P_k^{-1}\otimes p_k^*\R\Psi_{\cP_{A_k}}\cM_k)
              &\simeq &\R \Delta (p_k^*\R\Psi_{\cP_{A_k}}\cM_k) \otimes P_k\\
  &\simeq & p_k^*\big(\R\Delta (\R\Psi_{\cP_{A_k}}\cM_k)\big)\otimes P_k.
\end{eqnarray*}
We define $\cF_k:= \R\Delta(\R\Psi_{\cP_{A_k}}(\cM_k))$. Since $p_k^* \cF_k \otimes P_k$ is a direct summand of the coherent sheaf $\cW_{X/A}$, it follows that $\cF_k$ is a sheaf. Moreover,  $\R\Phi_{\cP_{A_k}}\R\Delta(\cF_k)[\dim A_k]\simeq (-1)^*_{A_k}\cM_k$ is a pure sheaf on $A_k$, hence $\cF_k$ is M-regular. This concludes the proof of Corollary \ref{decompositionwithtorsion}.
 %
\end{proof}

The following corollary is clear from the above theorem and will be used to prove a birational criterion for morphisms between varieties of maximal Albanese dimension in Section 5.
\begin{coro}\label{corollary-dualpure}Let $f: X\rightarrow A$ be a generically finite morphism. Assume that $\cQ$ is a direct summand of $f_*\omega_X$. Then
we have the decomposition for the sheaf $\widehat{\R\Delta(\cQ)}$:
$$\widehat{\R\Delta(\cQ)}\simeq \bigoplus_{P_k+\PA_k\in S(\cQ)}\cM_k, $$
where $\cM_k$ is a pure sheaf supported on $-P_k+\PA_k$.
\end{coro}

\begin{theo}Let $f: X\rightarrow A$ be a generically finite morphism. Then there exists an abelian Galois \'etale cover
$\pi_A: \widetilde{A}\rightarrow A$ with the base change $\widetilde{f}: \widetilde{X}:=X\times_A\widetilde{A}\rightarrow \widetilde{A}$ such that $K_{\widetilde{X}}$ is globally generated away from the exceptional locus of $\widetilde{f}$.
\end{theo}
\begin{proof}
By Theorem \ref{decompositionwithtorsion}, we have $\cW_{X/A}\simeq \bigoplus_{P_{k}+\PA_{k}\in S_X}P_{k}^{-1}\otimes p_{k}^*\cF_{k}$, where $\cF_k$ is M-regular on $A_k$. By \cite[Proposition 3.1]{D2}, there exists an abelian Galois \'etale cover $\pi_A: \widetilde{A}\rightarrow A$ such that $\pi_A^*(P_{k}^{-1}\otimes p_{k}^*\cF_{k})$ is globally generated. Considering the base change:
\begin{eqnarray*}
\xymatrix{
\widetilde{X}\ar@{}[drr]|\square\ar[d]^{\pi}\ar[rr]^{\widetilde{f}} && \widetilde{A}\ar[d]^{\pi_A}\\
X\ar[rr]_f && A,}
\end{eqnarray*}
we have $\cW_{\widetilde{X}/\widetilde{A}}=\pi_A^*\cW_{X/A}$. Therefore, $\cW_{\widetilde{X}/\widetilde{A}}$ is globally generated and so is $\widetilde{f}_*\omega_{\widetilde{X}}$. This implies that $K_{\widetilde{X}}$ is globally generated away from the exceptional locus of $\widetilde{f}$.
\end{proof}
\section{Criterion of birationality}
We consider a surjective and generically finite morphism between
smooth projective varieties $t: X\rightarrow Y$. We are interested
to know when $t$ is birational. When there exists a generically
finite morphism $p: Y\rightarrow A$ over an abelian varitey $A$,
there are several cohomological criterion about the birationality
of $t$, see for instance \cite[Theorem 3.1]{hp2}, and \cite[Lemma
5.4]{CDJ}.

In this section, we consider the case when $X\xrightarrow{t}
Y\xrightarrow{g} A$ are generically finite over the abelian
variety $A$. We denote by $f=g \circ t$. We will always assume that
\begin{equation}\label{equation1}V^0(\omega_X,
f)=\cup_{i=1}^N\PA_i,\end{equation} where $\PA_i$ is a proper
abelian subvariety of $\PA$, and in particular, we have
$\chi(X,\omega_X)=0$.

We may write
\begin{eqnarray}\label{def1}t_*\omega_X=\omega_Y\oplus
\cQ.\end{eqnarray} Then $t$ is birational if and only if $\cQ=0$.
Since $f_*\omega_X$ is a GV-sheaf, so is $g_*\cQ$. Hence $\cQ=0$ if
and only if  $V^0(\cQ, g)=\emptyset$.


For any irreducible component $\PA_i\subset V^0(\omega_X, f)$, we consider the Stein factorizations
\begin{eqnarray}\label{diagram1}
\xymatrix{
X\ar[r]^{t}\ar[d]^{h_{X_i}}& Y\ar[d]^{h_{Y_i}}\ar[r]^g & A\ar[d]^{\pi_i}\\
X_i\ar[r]^{t_i} & Y_i\ar[r]^{g_i} & A_i}
\end{eqnarray}
and the set
$$\Sigma_{b}:=\{ 1\leq j\leq N \mid t_j \text{ is  birational} \},$$ and its complementary set
$$\Sigma_{nb}:=\{ 1\leq j\leq N \mid t_j \text{ is not  birational} \}.$$
\begin{prop}\label{criteria}Under the above assumptions, it follows that
\begin{itemize}
\item[1)] for any $T\in S(g_*\cQ)$, $T\subseteq \cup_{j\in \Sigma_{nb}}\PA_j$, and in particular,
$$\bigcup_{T\in S(g_*\cQ)}T \subseteq\bigcup_{j\in \Sigma_{nb}}\PA_j;$$
\item[2)] if all $t_i$ are birational, $t$ is also birational.

\end{itemize}
\end{prop}

\begin{proof}
We first prove $1)$.
Assume that $S(g_*\cQ)\neq \emptyset$
and take $T\in S(g_*\cQ)\subset S(f_*\omega_X)$. By Simpson's theorem \cite{sim}, we can write $T=P+\PK$, where $P$ is a torsion point and $\PK$ is a abelian subvariety of $\PA$.
Assume $P+\PK\subset \PA_1$, it suffices to prove that $1\in \Sigma_{nb}$. 

Since $P\in\PA_1$, we may take an \'{e}tale cover
$A_1'\rightarrow A_1$ such that the pull-back of $P$ is trivial.
After base change by $A_1'\rightarrow A_1$, diagram
(\ref{diagram1}) now reads:
\begin{equation*}
\xymatrix{
X'\ar[r]^{t'}\ar[d]^{h'_X} & Y'\ar[d]^{h'_Y}\ar[r]^{g'} & A'\ar[d]^{\pi'}\\
X_1'\ar[r]^{t_1'} & Y_1'\ar[r]^{g_1'} & A_1'}
\end{equation*}
where all $h'_X, h'_Y$ are fibrations and $t_1'$ is birational if
and only if $t_1$ is birational.


We know that $P+\PK$ is  an irreducible component of $V^k(\cQ, g)$.
Then we consider the composition of morphisms $X'\xrightarrow{t'}Y'\xrightarrow{g'}A'\rightarrow A\rightarrow K$ and take smooth models of the Stein factorization of $X'\rightarrow K$ and $Y'\rightarrow K$. We get
\begin{eqnarray*}
\xymatrix{
X'\ar[r]^{t'}\ar[d]^{\mu_X}& Y'\ar[d]^{\mu_Y}\ar[r]^{g'} & A'\ar[d]^{\pi'_K} \\
Z_X\ar[r]^{t_K} & Z_Y\ar[r]^{g_K}  & K}
\end{eqnarray*}
We write $t'_*\omega_{X'}=\omega_{Y'}\oplus \cQ'$, where $\cQ'$ is the pull-back of $\cQ$ on $Y'$ and moreover, since the pull back of $P$ is trivial on $A'$, $\pi_K'^*\PK\hookrightarrow \PA'$ is an irreducible component of
$V^k(\cQ', g')$. We know that $R^i \mu_{Y*}\cQ'$ is a GV-sheaf on $Z_Y$ for each $i\geq 0$ by Hacon's version of generic vanishing theorem (see \cite[Corollary 4.2]{Hac}), we conclude that
$V^0(R^k\mu _{Y*}\cQ', g_K)=\PK$.

We then observe that
$$
\begin{array}{lll} t_{K*}\omega_{Z_X}&= t_{K*}R^k \mu_{X*}\omega_{X'}  &\textrm{by \cite[Proposition 7.6]{ko1}},  \\
                       &= R^k \mu_{Y*}t'_{*}\omega_{X'} &\\
                       &= \omega_{Z_Y}\oplus R^k \mu_{Y*}\cQ'.&
                       \end{array}$$

Therefore, $h^0(Z_X, \omega_{Z_X})>h^0(Z_Y, \omega_{Z_Y})$, we
conclude that $\deg t_K>1$. Notice that we have
\begin{eqnarray*}
\xymatrix{
X'\ar[r]^{t'}\ar[d]& Y'\ar[d]  \\
X_1'\ar[r]^{t_1'}\ar[d] & Y_1'\ar[d]\\
Z_X\ar[r]^{t_K} & Z_Y}
\end{eqnarray*}
where all the vertical morphisms are fibrations. Hence  $\deg
t_K>1$ implies that $\deg t'_1>1$. Therefore, $t'_1$ is not
birational. This implies that $1 \in \Sigma_{nb}$.

For $2)$, we note that it suffices to prove that $V^0(\cQ, g)=\emptyset$. Assume that $V^0(\cQ, g)$ is not empty, then it contains an irreducible component $T$ of codimension $k\leq \dim X$. By \cite[Proposition 3.15]{PP-GV}, $T$ is also an irreducible component of $V^k(\cQ, g)$. In particular, $T\in S(g_*\cQ)$. On the other hand, all $t_i$ are birational by assumption, thus $\Sigma_{nb}=\emptyset$. Therefore, we have by $1)$ that $S(g_*\cQ)=\emptyset$, which is a contradiction.
\end{proof}

It will be useful to consider birationality when $X$ dominates
more than one varieties.  Let's consider the following commutative
diagram:
\begin{eqnarray*}
\xymatrix{
X\ar[rr]^{f_1}\ar[d]_{f_2}\ar[drr]^{q}&& Y_1\ar[d]^{a_1}\\
Y_2\ar[rr]^{a_2} && A}
\end{eqnarray*}
where all the morphisms are generically finite and surjective. We then write
$f_{i*}\omega_X=\omega_{Y_i}\oplus \cQ_i$ for $i=1,2$. 

\begin{prop}\label{criteria3-2finite}
Suppose that
$S(a_{1*}\cQ_1) \cap S(a_{2*}\cQ_2)=\emptyset$, then either $f_1$ or $f_2$ is
birational.
\end{prop}
\begin{proof}
Let's denote by $r_i$ the degree of $f_i$ and denote by $s_i$ the degree of $a_i$. Certainly, $r_1s_1=r_2s_2=\deg q$.

We may assume that $f_1$ is not birational. Then $\cQ_1$ is a sheaf of rank $r_1-1>0$.
We note that $a_{i*}\cQ_1$ is a direct summand of $$q_*\omega_X=a_{i*}\omega_{Y_2}\oplus a_{i*}\cQ_2,$$
for $i=1, 2$. Thus $\widehat{\R\Delta (a_{i*}\cQ_i)}$ is a direct summand of $\widehat{\R\Delta(q_*\omega_X)}$.
By Corollary \ref{corollary-dualpure}, we have the decomposition for $\widehat{\R\Delta (a_{i*}\cQ_i)}$:
\begin{eqnarray}\label{decomsummand}
\widehat{\R\Delta (a_{i*}\cQ_i)}\simeq \oplus_{P_k+\PA_k\in S(a_{i*}\cQ_i)} \cF_k,
\end{eqnarray}
where $\cF_k$ is a pure sheaf supported on $-P_k+\PA_k$.

We now consider the morphisms:
$$\Id: a_{1*}\cQ_1\xrightarrow{ \Psi=(\Psi_1, \Psi_2)} a_{2*}\omega_{Y_2}\oplus a_{2*}\cQ_2 \rightarrow a_{1*}\cQ_1$$
We notice that
\begin{eqnarray*}&&\Hom_{\D(A)}(a_{1*}\cQ_1, a_{2*}\omega_{Y_2}\oplus a_{2*}\cQ_2)\\&\simeq &
\Hom_{\D(A)}\big(\R\Delta(a_{2*}\omega_{Y_2})\oplus \R\Delta(a_{2*}\cQ_2), \R\Delta(a_{1*}\cQ_1)\big)\\
&\simeq&\Hom_{\D(\PA)}\big(\widehat{\R\Delta(a_{2*}\omega_{Y_2})}\oplus \widehat{\R\Delta(a_{2*}\cQ_2}), \widehat{\R\Delta(a_{1*}\cQ_1)}\big).
\end{eqnarray*}

We denote by $\Phi_i$ the image of $\Psi_i$ under the above natural transformation.
By assumption $S(a_{1*}\cQ_1) \cap S(a_{2*}\cQ_2)=\emptyset$, thus by (\ref{decomsummand}), we have $\Phi_2=0$ and hence
$\Psi_2=0$.
Therefore, $ a_{1*}\cQ_1$ is a direct summand $ a_{2*}\omega_{Y_2} $.

We note that $h^n(A, a_{2*}\omega_{Y_2})=1$ but $h^n(A, a_{1*}\cQ_1)=h^n(A, q_{*}\omega_{X})-h^n(A, a_{1*}\omega_{Y_1})=0$. Hence $\rank a_{2*}\omega_{Y_2}>\rank a_{1*}\cQ_1.$

We then compare the rank of these two sheaves: $s_2>s_1(r_1-1)$. Hence
$r_1s_1=r_2s_2 >r_2s_1(r_1-1)$ and $r_2<\frac{r_1}{r_1-1}$. Therefore $r_2=1$ and $f_2$ is birational.
\end{proof}
Here is an application of Proposition \ref{criteria3-2finite} and will be used in the last section.
\begin{coro}\label{corollary-1}Under the setting of Proposition \ref{criteria},
suppose that all $t_i$ but possible one are birational, then $t$ is birational.
\end{coro}

\begin{proof} Assume that $t$ is not birational, then $\Sigma_{nb} \neq \emptyset$. We may assume that $\Sigma_{nb}=\{1\}$.

We define $Z:=X_1\times_{A_1}A$. Then the induced morphism $s: X\rightarrow Z$ is generically finite and surjective.

We now consider the commutative diagram:
\begin{eqnarray*}
\xymatrix{
X\ar[rr]^{t}\ar[d]_{s}\ar[drr]^{q}&& Y\ar[d]^{a_1}\\
Z\ar[rr]^{a_2} && A.}
\end{eqnarray*}
Write $t_*\omega_X=\omega_Y\oplus \cQ_Y$ and $s_*\omega_X=\omega_Z\oplus\cQ_Z$.
We note by Proposition \ref{criteria} that all elements of $S(a_{1*}\cQ_Y) $ are contained in $\PA_1$.
On the other hand, by the construction of $Z$, any element of $S(a_{2*}\cQ_Z) $ is not contained in $\PA_1$. Hence $S(a_{1*}\cQ_Y)\cap S(a_{2*}\cQ_Z)=\emptyset$.
We now apply Proposition \ref{criteria3-2finite} to conclude that either $t$ or $s$ is birational.
However,  $s$ cannot be birational because $Z$ is not of general type. This is the desired contradiction.
\end{proof}
\section{Primitive varieties of $\chi=0$}
We are now ready to  apply the results proved in previous sections to study the structure of a smooth projective variety $X$ of maximal Albanese dimension, of general type, and $\chi(X, \omega_X)=0$. In this section, we will formulate a conjecture about the structure of such varieties.

Inspired by the main theorem in \cite{CDJ}, we find the following definition useful.

\begin{defi}
Let $X$ be a smooth projective variety of maximal Albanese dimension, of general type, and $\chi(X, \omega_X)=0$. We say that $X$ is {\it primitive of $\chi=0$} if
\begin{itemize}
 \item[ ]for any non-trivial fibration $h: X\rightarrow Y$ to a normal projective variety with a general fiber $F$, we have $\chi(F, \omega_F)>0$;
\end{itemize}
\end{defi}

For a primitive variety of $\chi=0$, we know that $q(X)=\dim X=n$, the Albanese morphism $a_X: X\rightarrow A_X$ is generically finite
 and surjective, and $A_X$ has at least $3$ simple factors  (see \cite[Lemma 4.6 and Corollary 3.5]{CDJ}).
Moreover, we see from the definition that primitive varieties of $\chi=0$ are the building blocks of varieties of maximal Albanese dimension, of general type, and of $\chi(X, \omega_X)=0$. But more precisely, we have the following structural result.

\begin{prop}\label{reduction}Let $X$ be a variety of general type, of maximal Albanese dimension, and $\chi(X, \omega_X)=0$. Then either $X$ is primitive of $\chi=0$ or there exists an irregular fibration\footnote{Here, by definition, an irregular fibration $f: X\rightarrow Z$ is the Stein factorization of a morphism $g: X\rightarrow A$ to an abelian variety. Hence $Z$ is a normal projective variety and there exists a finite morphism $Z\rightarrow A$.} $f: X\rightarrow Z$  with a general fiber $F$  primitive of $\chi=0$.
\end{prop}
\begin{proof}

 If $X$ is not primitive of $\chi=0$, we take $f: X\rightarrow Y$ to be a fibration with a general fiber $F$ such that $\chi(\omega_F)=0$ and assume that $\dim F$ is minimal among all such fibrations.

 By Lemma \ref{reduction1} below and the minimality of $\dim F$, we see that $f$ is an irregular fibration and $a_{X}(F)$ is a translate of an abelian variety $K$ of $A_{X}$. We claim that $F$ is primitive of $\chi=0$.

 Otherwise, there exists a fibration of $F$, whose general fiber has $\chi=0$.  Considering the generically finite morphism $a_X|_F: F\rightarrow K$, we conclude again by Lemma \ref{reduction1} that there exists an abelian subvariety $K'$ of $K$ such that an irreducible component $F'$ of a general fiber of $F\rightarrow K\rightarrow K/K'$ has $\chi(\omega_{F'})=0$. Then considering the Stein factorization $X\xrightarrow{g}Z\rightarrow A_X/K'$, $F'$ is a general fiber of $g$, which is again a contradiction to the minimality of the dimension of $F$.
  \end{proof}
\begin{lemm}\label{reduction1}
Let $\alpha: X\rightarrow A$ be a generically finite morphism from a smooth projective variety of general type to an abelian variety.
Assume that we have a fibration $f: X\rightarrow Y$ with a general fiber  $F$ such that $\chi(\omega_F)=0$. Then there exists a quotient of abelian varieties $A\rightarrow B$ such that $f$ factors birationally through the Stein facorization $g: X\rightarrow Z$ of the induced morphism $X\rightarrow A\rightarrow B$. Moreover, we have $\chi(\omega_{F'})=0$ for a general fiber $F'$ of $g$.
 \end{lemm}
\begin{proof}
We consider the morphism $\alpha|_F: F\rightarrow A$. By \cite[Theorem 4.2]{ch2}, there exists an abelian subvariety $K$ of $A$ such that $\alpha|_F(F)$ is fibred by $K$ and moreover, an irreducible component $F'$ of a general fiber $F\rightarrow A\rightarrow A/K$ has $\chi(\omega_{F'})=0$. We then take $B$ to be $A/K$ and let $X\xrightarrow{g}Z\rightarrow B$ be the Stein factorization of $X\rightarrow B$. It is easy to check that $g$ is the irregular fibration we are looking for.
\end{proof}

Here we have some basic properties for primitive varieties of $\chi=0$.
\begin{lemm}\label{lemma1}Let $X$ be a primitive variety of $\chi=0$, then
\begin{itemize}
\item[1)]for any simple abelian sub-variety $\PK\hookrightarrow \PA_X$, there exists an irreducible component $\PA_i$ of $V^0(\omega_X,a_X)$ such that
the composition of morphisms $\PA_i\rightarrow \PA_X\rightarrow \PA_X/\PK$ is an isogeny;
\item[2)] for any simple abelian variety $\PK\hookrightarrow \PA_X$, the induced morphism $X\rightarrow A_X\rightarrow K$
is a fibration.
\end{itemize}

\end{lemm}
\begin{proof}
For $1)$, we consider the surjective morphism $g: X\rightarrow A_X\rightarrow K$. Since $X$ is primitive, for an irreducible component $F$ of a general fiber of $g$,
we have
$\chi(F, \omega_F)>0$. Hence the natural morphism $V^0(\omega_X,a_X)\rightarrow \PA_X\rightarrow \PA_X/\PK$ is surjective. Therefore there
exists an irreducible component $\PA_i$ of $V^0(\omega_X, a_X)$ such that the natural morphism $\PA_i\rightarrow \PA_X/\PK$ is an isogeny.

We argue by contradiction to prove $2)$. Assume that $X\rightarrow A$ is not a fibration, we take a modification of Stein factorization $X\rightarrow M\rightarrow K$ such that $M$ is a smooth projective variety. Then the morphism $M\rightarrow K$ is not an \'etale morphism and is of degree $\geq 1$.
Since $K$ is simple, we have $\chi(M, \omega_M)>0$. On the other hand, by 1), there is an irreducible component
$T_i$ of $V^0(\omega_X)$ such that $T_i\times \PK\rightarrow \PA_X$ is an isogeny. After taking an \'etale cover of $X$, we may assume that $T_i=\PA_i$ passes through the origin of $\PA_X$. Hence $X\rightarrow M\times Y_i$ is generically finite and surjective.
However, we then have $\chi(X, \omega_X)\geq \chi(M, \omega_M)\chi(Y_i, \omega_{Y_i})>0$, which is a contradiction.
\end{proof}
\begin{coro}\label{corollarycomponents}Let $X$ be primitive of $\chi=0$. Then,
\begin{itemize}
\item[(1)]if $A_X$ has $m$ simple factors, then $V^0(\omega_X)$ has at least $m$ irreducible components;
\item[(2)]if $A_X$ has $3$ simple factors, then each irreducible component $\PA_i$ of $V^0(\omega_X)$ has $2$ simple factors.
\end{itemize}
\end{coro}

We end this section with a conjectural characterization of primitive varieties of $\chi=0$.
\begin{conj}\label{conj}
Let $X$ be a smooth projective variety. Then $X$ is a primitive variety of $\chi=0$ if and only if there exist simple abelian varieties $K_1, K_2, \ldots, K_{2k+1}$, double coverings from normal projective varieties of general type $F_i\rightarrow K_i$ with involutions $\sigma_i$, and an isogeny $K_1\times\cdots\times K_{2k+1}\rightarrow A_X$ such that the base change $\widetilde{X}$ is birational to $$(F_1\times\cdots\times F_{2k+1})/\langle\sigma_1\times\cdots\times\sigma_{2k+1}\rangle.$$
\end{conj}
Together with Proposition \ref{reduction}, Conjecture \ref{conj} gives all possible structures for varieties of general type, of maximal Albanese dimension, and of $\chi(\omega_X)=0$.
We note one direction of Conjecture \ref{conj} is fairly standard but the other direction seems rather difficult.
In next section, we will prove Conjecture \ref{conj} when $A_X$ has only $3$ factors.
\section{Three simple factors}
In this section£¬ we assume that $X$ is of general type, of maximal Albanese dimension, with $\chi(X, \omega_X)=0$, and $A_X$ has only three simple factors.
 In particular, $X$ is primitive of $\chi=0$ (see \cite[Proposition 4.5]{CDJ}).
We are interested in the structure (up to \'etale covers and birational modifications) of such varieties.

We are free to take \'etale covers of $X$ and hence we always assume that each component of $V^0(\omega_X)$ passes through the origin of $\PA_X$ in this section.
We then write $V^0(\omega_X)=\cup_{i=1}^N\PA_i$. We know that the complementary of $\PA_i\hookrightarrow \PA_X$ is a simple abelian variety and $N\geq 3$.

The main theorem of this section is the following generalization of \cite[Theorem 5.1]{CDJ}.

\begin{theo}\label{miantheorem}
Let $X$ be a variety of general type, of maximal Albanese dimension, with $\chi(X, \omega_X)=0$, and
assume that $A_X$ has only three simple factors. There exist
simple abelian varieties $K_1$, $K_2$, $K_3$, double coverings
from normal varieties
 $F_i\rightarrow K_i$ with associated involution $\tau_i$, and an isogeny
 $\eta \colon K_1 \times K_2 \times K_3 \to A_X$, such that the base change $\tilde X$ is birational
 to $$(F_1\times F_2\times F_3)/\langle \sigma \rangle$$
 where $\sigma=\tau_1\times\tau_2\times\tau_3$ is the diagonal involution.
 That is, we have the following commutative diagram:

$$
\xymatrix@C=15pt
 {&(F_1\times F_2\times F_3)/\langle \sigma \rangle \ar@{->>}[dr]\\
\widetilde X\ar@{->>}[rr]^{a_{\widetilde X}}\ar@{}[drr]|\square\ar@{->>}[ur]^-\eps\ar@{->>}[d]&& K_1\times K_2\times K_3\ar@{->>}[d]^\eta\\
X\ar@{->>}[rr]_{a_X} &&A_X. }
$$
where  $\eps$ is a desingularization.

\end{theo}

Let $X\xrightarrow{f_i} Y_i\xrightarrow{t_i} A_i$ be a  modification of the Stein factorization of the morphism $X\xrightarrow{a_X}A_X\xrightarrow{p_i} A_i$ such that $Y_i$ is smooth projective, and denote by $a_{X*}\omega_X=\cO_{A_X}\oplus\cW_{X/A_X}$ and $t_{i*}\omega_{Y_i}=\cO_{A_i}\oplus\cF_i$, where $\cF_i$ is M-regular by Lemma \ref{lemma1}.

Since $A_X$ has only three simple factors and $X$ is primitive, by Lemma \ref{lemma1} 2), we know that $$S_X=\{\PA_i\mid 1\leq i\leq N\}.$$
By Theorem \ref{decomposition}, we have the following proposition, which is our starting point.
\begin{prop}\label{degrees}We have $\cW_{X/A_X}\simeq\bigoplus_{i=1}^Np_i^*\cF_i$. In particular, $$\deg a_X-1=\sum_{i=1}^N(\deg t_i-1).$$
\end{prop}


\subsection{Characterization of special primitive varieties}
In this subsection, we are going to prove  Theorem \ref{miantheorem} for
special primitive varieties of $\chi=0$ satisfying :
\begin{itemize}
\item[(\dag 1)]  $A_X=K_1\times K_2\times K_3$.

\item[(\dag 2)] $V^0(\omega_X)$ contains three components
$\PA_1=\{0\}\times\PK_2\times\PK_3$,
$\PA_2=\PK_1\times\{0\}\times\PK_3$, and
$\PA_3=\PK_1\times\PK_2\times\{0\}$, and

\item[(\dag 3)]the induced morphism $X\xrightarrow{}
Z_k:=Y_i\times_{K_k} Y_j$ is birational, for $\{i, j, k\}=\{1, 2,3\}$.

\end{itemize}

Recall that we have generically finite morphism $Y_i \to A_i$ and
the induced fibrations $h_{ij} \colon Y_i \to K_j$. These fit
into the following diagram
\begin{eqnarray}\label{diagram2}
\xymatrix{Z_1 \ar[d]_{g_{12}}\ar[drr]^{g_{13}} && Z_2 \ar[d]_(.3){g_{23}}\ar[drr]_{g_{21}} && Z_3 \ar[d]^{g_{31}} \ar[dllll]_(.2){g_{32}}\\
Y_2 \ar[d]_{h_{23}}\ar[drr]^{h_{21}} && Y_3 \ar[d]_(.3){h_{31}}\ar[drr]_{h_{32}} && Y_1 \ar[d]^{h_{12}} \ar[dllll]_(.2){h_{13}}\\
K_3  && K_1  && K_2,
}
\end{eqnarray}
where $ Z_i $ is the main component of the fiber product over $K_i$. We note that,
by \cite[Theorem 2.3]{CH}, $Z_i$ is
of general type.

We denote $a_i=\deg t_i=\deg(Y_i/A_i)$ for $i=1, 2, 3$. Then by
assumption $(\dag 3)$, $\deg a_X=a_ia_j$ for any $1\leq i\neq
j\leq 3$. Hence $a:=a_1=a_2=a_3$.

\begin{lemm}\label{isotrivial-fibers}
There are smooth varieties of general type $F_1, F_2, F_3$
generically finite over $K_1, K_2, K_3$ respectively such that
the fibration $h_{ij}: Y_i\rightarrow K_j$ is isotrivial with a
general fiber $F_k$, for $\{i, j, k\}=\{1, 2, 3\}$.
\end{lemm}

\begin{proof}We will show that  $h_{13}: Y_1\rightarrow K_3$ is isotrivial and the general fiber of $h_{13}$
 is birational to a general fiber of $h_{31}: Y_3\rightarrow K_1$. The same argument works for other fibrations.

Since $X\xrightarrow{}  Z_1$ is birational, there exists a
dominant map $\tau: Z_1 \dashrightarrow X \rightarrow Y_1$ fits
into  the following commutative diagram:
\begin{eqnarray*}
\xymatrix{
X\ar[rr]^{}\ar[drrr]_{f_1} && Z_1 \ar@{.>}[dr]|-{\tau}\ar[rr]^{(f_{12}, f_{13} )}&& K_2\times K_3\\
&&& Y_1\ar[ur]_{(h_{12}, h_{13})}},
\end{eqnarray*}
where $f_{12}=h_{32}\circ g_{13}$ and $f_{13}=h_{23}\circ g_{12}$.

Fix a general point $t\in K_1$, and denote respectively by
$F_{2t}$ and $F_{3t}$ the fibers of $h_{31}$ and $h_{21}$ over
$t$. Then the fiber of $Z_1 \rightarrow K_1$  is isomorphic to
$F_{3t}\times F_{2t}$ with the commutative diagram
\begin{eqnarray}\label{diagram3}
\xymatrix{& F_{2t} \times F_{3t}\ar[dl] \ar[dr] \ar@{.>}[d]|-{\tau} & \\
F_{2t}\ar[dd]_{h_{32}}  & Y_1\ar[d] \ar[ddl]_{h_{12}} \ar[ddr]^{h_{13}}& F_{3t}\ar[dd]^{h_{23}}  \\
& K_2  \times K_3  \ar[dl] \ar[dr]& \\
K_2 && K_3\\
 }
\end{eqnarray}

Since $h_{21}$ is a fibration, we know that
$$ \deg(F_{3t}/K_3)=\deg (Y_2/A_2) =a.$$

Let $F_2$ be a general fiber of $h_{13}: Y_1 \to K_3$.  We
also have that $\deg(F_2/K_2)= \deg(Y_1/K_2\times K_3)=a$. Let $W$ be the main component of $ Y_1
\times_{K_3} F_{3t}$. From the right part of diagram (\ref{diagram3}), we see that there is a induced
rational dominant map $\tau': F_{2t} \times F_{3t} \dashrightarrow
W$ over $Y_1$. Since
$$\left\{\begin{array}{ll} \deg(F_{2t} \times F_{3t}/Y_1) & =a^2/
\deg(Y_1/K_2 \times K_3)=a,\\
 \deg(W/Y_1) &=\deg(F_{3t}/K_3)=a.
 \end{array} \right.
 $$
It follows that $\tau'$ is birational. Thus $h_{13}: Y_1 \to K_3$
is isotrivial with a general fiber $F_2$ birational to $F_{2t}$
for any general $t$.
\end{proof}

\begin{lemm}\label{3components}
 $V^0(\omega_X, a_X)$ has exactly
three irreducible components $\PA_1$, $\PA_2$, and $\PA_3$.
\end{lemm}

\begin{proof}
Let $F_{ij}:=F_i \times F_j$ and $F_{123} :=F_1 \times F_2 \times
F_3$. Since there are dominant generically finite rational map
$F_{12} \dashrightarrow Y_3$ (resp. $F_{13} \dashrightarrow Y_2$,
$F_{23} \dashrightarrow Y_1$), it is straightforward to see that
there exists dominant generically finite rational map $F_{123}
\dashrightarrow Z_k$ for $k=1,2,3$. In particular, they induce a
dominant generically finite rational map $F_{123} \dashrightarrow
X$. Resolve the indeterminancy via $\nu: X' \to F_{123}$, we have
a generically finite morphism $\rho \colon X' \to X$.

We claim that $V^1(X', a_X \circ \rho)$ consists of finite unions
of translation of $\PA_1$, $\PA_2$, and $\PA_3$. Since
$\chi(\omega_X)=0$,
$$V^0(\omega_X, a_X) =V^1(\omega_X, a_X) \subset V^1(\omega_{X'},
a_X \circ \rho)$$ and each component of $V^0(\omega_X, a_X)$ is a
non-simple abelian subvariety passing through the origin, we
conclude the Lemma.

It remains to prove the claim. Let $p=(\alpha, \beta, \gamma):
F_{123} \to A_X=K_1 \times K_2 \times K_3$ be the given surjective
morphism. Clearly, $\nu$ is birational and $p \circ \nu=a_X \circ
\rho$. Hence
$$V^1(\omega_{X'}, a_X \circ \rho)=V^1(\omega_{X'}, p \circ
\nu)=V^1(\omega_{F_{123}}, p).$$

Since $K_1, K_2$ and $K_3$ are simple abelian varieties,
$V^1(\omega_{F_1},\alpha)$, $V^1(\omega_{F_2},\beta)$ and
$V^1(\omega_{F_3}, \gamma)$ are finite union of isolated points.
By K\"unneth formula,

\begin{eqnarray*}V^1(\omega_{F_{123}}, p)
&=&\big(V^1 (\omega_{F_1},\alpha) \times \PK_2 \times \PK_3
\big)\bigcup
\big(\PK_1\times V^1(\omega_{F_2}, \beta)\times\PK_3\big)\\
&&\bigcup \big(\PK_1\times\PK_2\times
V^1(\omega_{F_3},\gamma)\big).
\end{eqnarray*}
This verifies the claim.
\end{proof}

By Lemma \ref{3components},
$V^0(\omega_X)=\PA_1\cup\PA_2\cup\PA_3$.  Thus we apply
Proposition \ref{degrees} and get $a^2-1=3(a-1)$. Hence $a=2$ and
$\deg(Y_i/A_i)=\deg(F_i/K_i)=2$, for $i=1, 2, 3$.

Replace $F_i$ by its Stein factorization over $K_i$, we may and do
assume that $F_i$ is normal and double cover over $K_i$. Let
$\tau_i$ be the corresponding involution.

The covering $F_2 \times F_3 \to K_2 \times K_3=A_1$ is Galois
with Galois group $\mathbb{Z}_2 \times \mathbb{Z}_2$. The function
field $K(Y_1)$ is an intermediate field of the extension $K(F_2
\times F_3)/K(A_1)$ and of degree $2$ over $K(A_1)$. Together with
the fact that $Y_1$ is of general type, it follows that $Y_1$ is
birational to $(F_1 \times F_2)/\langle\tau_1\times\tau_2\rangle$ by exhausting
all intermediate fields.

It is clear to see that $Y_2, Y_3$ has the same structure. In
fact, the similar argument also shows that $X$ is birational to $
(F_1\times F_2\times F_3)/\langle\tau_1\times\tau_2\times\tau_3\rangle$.

We thus conclude this subsection that
\begin{prop}\label{mainprop} Under the hypothesis $\dag $, there exist simple abelian
varieties $K_1$, $K_2$, $K_3$, double coverings from normal
varieties
 $F_i\rightarrow K_i$ with associated involution $\tau_i$, such that  $ X$ is birational
 to $$(F_1\times F_2\times F_3)/\langle \sigma \rangle$$
 where $\sigma=\tau_1\times\tau_2\times \tau_3$ is the diagonal involution.
\end{prop}

\subsection{The general case}
We prove the main theorem in this subsection. To start with,
it is convenient to introduce the following notions.

 \begin{defi}
 A primitive variety of $\chi=0$ is said to be {\it special} if it is birational to $(F_1\times F_2\times F_3)/\langle\tau_1\times\tau_2\times \tau_3\rangle$ as in Proposition \ref{mainprop}. A primitive variety $X$ of $\chi=0$
 is said to be {\it quasi-special} if there is an \'etale base change $\tilde{X} \to X$ so that $\tilde{X}$ is special.

 Given a variety $X$ primitive of $\chi=0$, we said that $X$ is {\it minimal} if $X$ is minimal among
  smooth projective varieties of general type sitting
between $X$ and $A_X$ (up to birational equivalent). More precisely,  let $X_0$ be a variety of general type sit between $X$ and $A_X$, then $X \to X_0$ is birational.
 \end{defi}

Fix a primitive variety $X$ with $\chi=0$.
For any component $\PA_i$, we have an map $t_i \colon Y_i \to A_i$. Let $d_i:=\deg(t_i)$.

For any two distinct components $\PA_i, \PA_j$, let $\PK_{ij}$ be the neutral component of $\PA_i\cap\PA_j$. We consider $Z_{ij}$ a desingularization of an irreducible component of the main component of $(Y_i\times_{K_{ij}}Y_j) \times_{A_i \times_{K_{ij}} A_j} A_X$. Replacing $X$ by its higher model, we may assume that there exists induced maps $\rho_{ij} \colon X \to Z_{ij}$ and $a_{ij} \colon Z_{ij} \to A_X$. It is easy to see the following properties of $Z_{ij}$:
\begin{itemize}
\item[1)] $Z_{ij}$ is of general type.
\item[2)] $\deg(Z_{ij}/A_X)| d_i d_j$.
\item[3)] $V^0(\omega_{Z_{ij}}, a_{ij}) \supset \PA_i, \PA_j$ and we have the natural morphisms $Z_{ij}\rightarrow Y_i\rightarrow A_i$ and $Z_{ij}\rightarrow Y_j\rightarrow A_j$. In particular, $Z_{ij}$ is not quasi-special if $d_i \ge 3$ or $d_j \ge 3$.
\end{itemize}


\begin{lemm} \label{etale} Pick  any three irreducible components, say $\PA_1$, $\PA_2$, and  $\PA_3$ of $V^0(\omega_X)$.
 Assume that all
 $\rho_{12}, \rho_{13}, \rho_{23}$ are  birational.

 Then  after an abelian \'etale base change
 $\tilde{X} \to X$, one may assume that $\tilde{X}$ satisfying $\dag 1, \dag 2,
\dag 3$. Hence $X$ is quasi-special.
 \end{lemm}

 \begin{proof}
By \cite[Proposition 4.3]{CDJ}, we know that after an abelian
\'etale cover $\widetilde{A}_X\rightarrow A_X$ and make the base
change:
 \begin{eqnarray*}
\xymatrix{
\widetilde{X}\ar[d]\ar[r]^{a_{\widetilde{X}}}& \widetilde{A}_X\ar[d]^{\pi}\\
X\ar[r]^{a_X}& A_X,}
\end{eqnarray*}
 we may assume that $\widetilde{A}_X=K_1\times K_2\times K_3$ and $\widehat{\widetilde{A_1}}=\widehat{\pi}(\PA_1)=\{0\}\times\PK_2\times\PK_3$, $\widehat{\widetilde{A_2}}=\widehat{\pi}(\PA_2)=\PK_1\times\{0\}\times\PK_3$, and $\widehat{\widetilde{A_3}}=\widehat{\pi}(\PA_3)=\PK_1\times\PK_2\times\{0\}$. As before, for $i=1,2,3$, we denote by  $\widetilde{X}\xrightarrow{\widetilde{g}_i} \widetilde{Y}_i\xrightarrow{\widetilde{t}_i} \widetilde{A_i}$ a modification of the Stein factorization of the natural morphism $\widetilde{X}\rightarrow \widetilde{A_i}$ such that $\widetilde{Y}_i$ is smooth projective.

 We claim that the induced morphism  $\widetilde{X}\rightarrow \widetilde{Y}_i\times_{K_k}\widetilde{Y}_j$
 is birational for
 $\{i, j, k\}=\{1,2,3\}$. To see this, note that since $\rho_{ij}$ is birational, we have $\deg a_X|(d_id_j)$, for any $i, j\in\{1,2,3\}$.

Thus $\deg a_{\widetilde{X}}=\deg a_X|(d_id_j)$.
 We also know that $t_i$ and $\widetilde{t}_i$ are respectively the Albanese morphisms of $Y_i$ and $\widetilde{Y}_i$. Thus $\deg\widetilde{t}_i\geq d_i$. Thus $\deg\big((\widetilde{Y}_i\times_{K_k}\widetilde{Y}_j)/\widetilde{A}_X\big)\geq d_id_j$ for $\{i, j, k\}=\{1,2,3\}$. Hence the induced morphism $\widetilde{X}\rightarrow \widetilde{Y}_i\times_{K_k}\widetilde{Y}_j$ is birational.
 \end{proof}

\begin{prop}\label{quasi-special} We have the following criterion for quasi-special primitive varieties with $\chi=0$:
\begin{enumerate}
\item if $X$ is minimal, then $X$ is quasi-special.

\item if $V^0(\omega_X, a_X)$ has three components, then $X$ is quasi-special.

\item if $\deg(X/A_X)=4$, then $X$ is quasi-special.
\end{enumerate}
\end{prop}

\begin{proof}
$(1).$ Suppose that $X$ is minimal. We pick three components $\PA_1$, $\PA_2$, and $\PA_3$. For $(i,j)=(1,2), (1,3)$ and $(2,3)$,  $Z_{ij}$ of general type sitting  between $X$ and $A_X$
and hence is
 birational to $X$ by minimality of $X$.
Therefore $X$ is quasi-special by Lemma \ref{etale}.

$(2).$ For each $(i,j)=(1,2), (1,3)$ and $(2,3)$,
  we have the commutative diagram:
 \begin{eqnarray*}
 \xymatrix{
 X  \ar[r]^{\rho_{ij}}\ar[d] & Z_{ij}\ar[r]\ar[d]& A_X\ar[d]\\
 Y_t\ar@{=}[r] & Y_t\ar[r] & A_t,}
 \end{eqnarray*}
 where $t=i$ or $j$.

 Since $V^0(\omega_X)$ consists of three components $\PA_1$, $\PA_2$, and $\PA_3$, by Corollary \ref{corollary-1}, $\rho_{ij}$ is birational. Thus by Lemma \ref{etale}, $X$ is quasi-special

$(3).$ Again let $X_0$ be a minimal primitive variety dominated by $X$. Since $X_0$ is quasi-special, $\deg(X_0/A_X)=4$. It follows immediately that $\deg(X/X_0)$ $=1$ and hence $X$ is quasi-special.
\end{proof}

 \begin{prop} It $X$ is primitive of $\chi=0$ such that $A_X$ has three simple factors, then $X$ is quasi-special.
 \end{prop}
 \begin{proof}
 Suppose on the contrary that $X$ is not quasi-special. Since minimal primitive varieties are quasi-special,
 there exists $X'$ sitting between $X$ and $A_X$ such that $X'$ is not quasi-special but any other general type variety dominated by  $X'$ is quasi-special. Replace $X$ by $X'$, we may and do assume that $X$ is not quasi-special,
 but any variety of general type dominated by $X$ and not birational to $X$ is quasi-special.

For $X\xrightarrow{f_i}Y_i\xrightarrow{t_i} A_i$, we denote by $d_i$ the degree of $t_i$, for each $A_i\in S_X$. We may assume that $d_1\geq\cdots\geq d_N\geq 2$. Since $X$ is not quasi-special, by Proposition \ref{quasi-special}, we have $N\geq 4$. We will deduce a contradiction by the following steps.

\medskip
\noindent{\bf Step 1.} Either $d_1>d_2=\cdots d_N=2$, or $d_1=\cdots=d_N=2$.

First of all, we see that the set $T=\{d_i\mid 1\leq i\leq N\}$ contains at most two different numbers. Otherwise, we may assume that $d_i>d_j>d_k$. Then $Z_{ik}$, $Z_{ik}$, $Z_{jk}$ are of general type. They are not all birational to $X$ by Lemma \ref{etale}, since $X$ is not quasi-special. We may assume that $Z_{jk}$ is not birational to $X$ and hence $Z_{jk}$ is quasi-special. Then $d_j=d_k=2$, which is a contradiction.

If $T$ contains two different numbers, we may assume that $d_1=\cdots=d_s>d_{s+1}=\cdots=d_N\geq 2$. If $i \le s$, then  $Z_{ij}$ can not be quasi-special and thus birational to $X$.
 If $s\geq 2$, we consider $Z_{12}$. Take a  component  $\PA_i$ of $V^0(\omega_{Z_{12}})$ for some $i\geq 3$. Then $X$ is birational to $Z_{12}$, $Z_{1i}$ and $Z_{2i}$ and hence $X$ is quasi-special by Lemma \ref{etale}, which contradicts the assumption.
 Hence $s=1$. Moreover, since $Z_{1i}$ is birational to $X$ for any $i\geq 2$, $Z_{2j}$ should be quasi-special for any $2\leq j\leq N$, we then have $d_2=\cdots=d_N=2$.

We next consider that $d:=d_1=\cdots=d_N$. If $d\geq 3$, then $Z_{ij}$ is not quasi-special and hence is birational to $X$  for any $1\leq i<j\leq N$. But by Lemma \ref{etale}, $X$ is quasi-special, which is a contradiction. Thus $d_1=\cdots=d_N=2$, and each $Z_{ij}$ is quasi-special.

\medskip
\noindent{\bf Step 2.} In both cases, $Z_{23}$ is quasi-special and we may assume that $S_{Z_{23}}=\{\PA_1, \PA_2, \PA_3\}$.

Since $\deg(Z_{23}/A_X)=4$, by lemma \ref{quasi-special}, $Z_{23}$ is always quasi-special. Moreover, since $N\geq 4$, in the second case $d_1=\cdots=d_N=2$,
we may assume that $S_{Z_{23}}=\{\PA_1,\PA_2, \PA_3\}$ and thus $\PA_4\notin S_{Z_{23}}$.

Assume that $d_1>d_2=\cdots=d_N=2$, then each $Z_{1i}$ is not quasi-special and hence is birational to $X$, for any
$2\leq i\leq N$. Thus $\deg(X/A_X)=2d_1$. On the other hand, $Z_{23}$ is quasi-special.
We claim that $S_{Z_{23}}=\{\PA_1,\PA_2, \PA_3\}$. Otherwise, the composition of morphisms
$Z_{23}\rightarrow A_X\rightarrow A_1$ is a fibration.
Thus the main component of $Z_{23}\times_{A_1}Y_1$ is irreducible and is of degree $4d_1$ over $A_X$, which is absurd.
Thus we again have $\PA_4\notin S_{Z_{23}}$.

\medskip
\noindent{\bf Step 3.} Deduce a contradiction.

By step 2, in both cases, we may assume that $S_{Z_{23}}=\{\PA_1,\PA_2, \PA_3\}$. Then since $Z_{23}$ is quasi-special, after an \'etale cover of $Z_{23}$,
we may assume that $Z_{23}$ is birational to $(F_1\times F_2\times F_3)/\sigma$ as in Proposition \ref{mainprop}.

Consider $\PK_{14}$, $\PK_{24}$, and $\PK_{34}$, we may assume that $\PK_{14}$ is different from $\PK_1\times 0\times 0$, $0\times \PK_2\times 0$, and $0\times 0\times\PK_3$ as an ablelian subvariety
of $\PA_X=\PK_1\times \PK_2\times \PK_3$.

The corresponding fiber product of $(F_2\times F_3)/\langle\tau_2\times \tau_3\rangle$ and $Y_4$ over $K_{14}$ is quasi-special.
To summarize, we have the following commutative diagram of morphisms
\begin{eqnarray*}
 \xymatrix{
 F_2\times F_3\ar[drr]_f\ar[r]^(.3){\varphi}& (F_2\times F_3)/\langle\tau_2\times\tau_3\rangle\ar[dr]^g\ar[r] & A_1\ar[d]\\
 && K_{14},}
\end{eqnarray*}
where $\varphi$ is the quotient morphism and hence is \'etale in codimension $1$. Hence $f$ is also isotrivial. For $t\in K_{24}$ general,
the corresponding fiber $F_t$ of $f$ is birational to a fixed variety $F$. We denote by $p$ and $q$ the natural projections of $F_1\times F_3$ to $F_1$ and $F_3$
and denote by $p_t: F_t\rightarrow F_1$ and $q_t: F_t\rightarrow F_3$ the restrictions of $p$ and $q$ on $F_t$. Both $p_t$ and $q_t$ are dominate surjective morphisms
between varieties of general type. By \cite[Corollary 2.2]{gp}, we know that almost all $p_t$  are  birational equivalent to each other and so is $q_t$.
However, then  for $t\in K_{24}$ general, the image of $(p_t, q_t): F_t\rightarrow F_1\times F_3$ is contained in a fixed proper Zariski subset of $F_1\times F_3$.
This is absurd.
 \end{proof}

\end{document}